\documentclass[12pt, a4paper]{article}

\usepackage{authblk}

\usepackage{amsmath, amssymb}
\usepackage{amsthm}

\makeatletter
	
	\@addtoreset{equation}{section}
\makeatother

\newtheorem{theorem}{Theorem}[section]

\newtheorem{lemma}[theorem]{Lemma}
\newtheorem{corollary}[theorem]{Corollary}
\theoremstyle{definition}

\title{Univalence and holomorphic extension \\ of the solution to \\ $\omega$-controlled Loewner--Kufarev equations}
\author[a]{Takafumi Amaba}
\author[b]{Roland Friedrich}
\author[c]{Takuya Murayama}
\affil[a]{Fukuoka University, 8-19-1 Nanakuma, J\^onan-ku, Fukuoka, 814-0180, Japan. \texttt{fmamaba@fukuoka-u.ac.jp}}
\affil[b]{Saarland University, Faculty of Mathematics, D-66123 Saarbr\"ucken, Germany. \texttt{friedrich@math.uni-sb.de}}
\affil[c]{Kyoto University, Department of Mathematics, Kyoto 606-8502, Japan. JSPS Research Fellow. \texttt{murayama@math.kyoto-u.ac.jp}}
\date{}

\begin{document}

\maketitle

\begin{abstract}
We prove that a solution to the $\omega$-controlled Loewner--Kufarev equation, which was introduced by the first two authors, exists uniquely, is univalent
and starlike
on the unit disk and can be extended holomorphically across the unit circle.


\end{abstract}

\section{Introduction}
\label{sec:intro}

In various fields of mathematics, univalent functions play important roles.
Not just are they fundamental objects in geometric function theory and in Teichm\"uller theory but also have deep connections with conformal field theory, integrable systems and even with random matrices.
The second author~\cite{Fr10} gave a concise picture of such connections in terms of the \emph{Schramm--Loewner evolution} and \emph{(Sato--)Segal--Wilson Grassmannian}, and then Markina and Vasil'ev~\cite{MV10, MV16} proposed an extension of his approach introducing the \emph{alternate Loewner--Kufarev equation}.

In order to generalize the above results further, the first two authors introduced the notion of \emph{controlled Loewner--Kufarev equations}
\[
df_t(z) = z f_t'(z) \, \{ dx_0(t) + d\xi(\mathbf{x}, z)_t \}, \quad t \in [0, T],\ f_0(z) = z \in \mathbb{D},
\]
in their previous paper~\cite{AF18}.
Here $\mathbf{x} = (x_1, x_2, \ldots)$, $\xi(\mathbf{x}, z)_t = \sum_{n=1}^{\infty} x_n(t) z^n$, and $x_n(t)$, $n \geq 0$, are complex-valued continuous functions of bounded variation.
As is described in \cite[Section~3.2]{AF18},
the solution $(f_t)_{0 \leq t \leq T}$ is embedded into the Segal--Wilson Grassmannian through Krichever's construction if it is univalent on the unit disk $\mathbb{D}$ and extends holomorphically across $\partial \mathbb{D}$.
Moreover, the same authors gave a sufficient condition~\cite[Theorem~3.3]{AF19} for the embedded solution to be continuous as a curve in the Grassmannian.
In that theorem, they consider the case in which the driving functions $x_0$ and $\mathbf{x}$ are \emph{controlled by a control function} $\omega \colon \{\, (s,t) \mid 0 \leq s \leq t \leq T \,\} \to \mathbb{R}_{+}$ \cite[Definition~3.2]{AF19} which satisfies $\omega(0, T) < 1/8$.

It should be noted that, in the paper~\cite{AF19}, it is a priori assumed that $f_t$ is univalent on $\mathbb{D}$ and extends to a holomorphic function on an open neighbourhood of $\overline{\mathbb{D}}$.
However, we expect that this property is intrinsic to a large class of controlled Loewner--Kufarev equations.
The purpose of the present article is to confirm this hypothesis in the $\omega$-controlled case.
The goal is the following:

\begin{theorem}
\label{thm:main}
Let $\alpha$ be the smallest real solution to the quartic equation
$2x^{4} - 8x^{3} + 11 x^{2} - 10 x + 2 = 0$
{\rm (}then $\alpha$ is positive{\rm )},
and suppose that
$
\omega(0, T) <
\alpha/2
$.
Then a holomorphic solution $(f_t)_{0 \leq t \leq T}$ to the $\omega$-controlled Loewner--Kufarev equation exists uniquely.
Moreover, $f_t$ is univalent
and starlike
on $\mathbb{D}$
and extends to a holomorphic function on an open neighbourhood of $\overline{\mathbb{D}}$ for each $t \in [0, T]$.
\end{theorem}

The rest of this article is organized as follows:
In Section~\ref{sec:setting}, we recall the basic concepts for our argument.
The assumptions on the driving functions are mentioned in Section~\ref{sec:driver}, and then the definition of a solution to a controlled Loewner--Kufarev equation is given in Section~\ref{sec:LKeq}.
Although these concepts appear in the previous papers~\cite{AF18, AF19}, we summarize them in order to recollect the terminology.
In Section~\ref{sec:main}, we prove Theorem~\ref{thm:main} in four steps, which consist of the uniqueness (Theorem~\ref{thm:uniqueness}), existence (Theorem~\ref{thm:existence}), holomorphic extension (Corollary~\ref{cor:hol_ext}),
univalence
and starlikeness
(Theorem~\ref{thm:univalence}) of the solution.

We use the following notation:
For a continuous function $F \colon [0, T] \to \mathbb{C}$ of bounded variation, the atomless measure $dF$ on $[0, T]$ is defined by the relation $dF((a, b]) = F(b) - F(a)$.
Its total variation is denoted by $\lvert dF \rvert$.

\section{Setting}
\label{sec:setting}

In this section, we describe our setting throughout this paper.

\subsection{Driving functions}
\label{sec:driver}

The \emph{driving functions} $x_0 \colon [0, T] \to \mathbb{R}$ and $x_n \colon [0, T] \to \mathbb{C}$, $n \geq 1$, are continuous functions of bounded variation.
For $\mathbf{x} := (x_1, x_2, \ldots)$, we define the formal power series
\[
\xi(\mathbf{x}, z)_t := \sum_{n=1}^{\infty} x_n(t) z^n
\]
and assume the following:
\begin{enumerate}
\item \label{cond:x0}
$x_0(0) = 0$;
\item \label{cond:xi0}
The series $\xi(\mathbf{x}, z)_0$ has convergence radius one;
\item \label{cond:tot_var}
$\sum_{n=1}^{\infty} \lvert dx_n \rvert ([0, T]) r^n$ converges for all $r \in (0, 1)$.
\end{enumerate}
We note that, for each $z \in \mathbb{D}$, the series $\sum_{n=1}^{\infty} dx_n(t) z^n$ converges with respect to the total variation norm on the space of complex measures on $[0, T]$ from the third condition\footnote{A slightly stronger condition is assumed in \cite[Definition~2.1~(2)]{AF18} to compute the Faber polynomials and Grunsky coefficients. For our purpose, the present condition is sufficient.}.

\begin{lemma}
\label{lem:xi}
Under the assumptions~\eqref{cond:x0}--\eqref{cond:tot_var} above, the series $\xi(\mathbf{x}, z)_t$ has convergence radius one for each $t \in [0, T]$.
The family $(\xi(\mathbf{x}, z)_t)_{0 \leq t \leq T}$ of holomorphic functions on $\mathbb{D}$ is continuous in the topology of locally uniform convergence.
Moreover, the function $t \mapsto \xi(\mathbf{x}, z)_t$ is of bounded variation and satisfies
\begin{equation} \label{eq:d_xi}
d\xi(\mathbf{x}, z)_t = \sum_{n=1}^{\infty} dx_n(t) z^n
\end{equation}
for each $z \in \mathbb{D}$.
\end{lemma}

\begin{proof}
For any $t \in [0, T]$ and $r \in (0, 1)$, we have
\begin{align*}
\sum_{n=1}^{\infty} \lvert x_n(t) \rvert r^n &\leq \sum_{n=1}^{\infty} \lvert x_n(t) - x_n(0) \rvert r^n + \sum_{n=1}^{\infty} \lvert x_n(0) \rvert r^n \\
&\leq \sum_{n=1}^{\infty} \lvert dx_n \rvert ((0, t]) r^n + \sum_{n=1}^{\infty} \lvert x_n(0) \rvert r^n < \infty.
\end{align*}
Hence $\xi(\mathbf{x}, z)_t$ has convergence radius one.

Similarly, for $0 \leq s \leq t \leq T$ and $r \in (0, 1)$, we get
\[
\sup_{\lvert z \rvert \leq r} \lvert \xi(\mathbf{x}, z)_{t} - \xi(\mathbf{x}, z)_s \rvert \leq \sum_{n=1}^{\infty} \lvert x_n(t) - x_n(s) \rvert r^n \leq \sum_{n=1}^{\infty} \lvert dx_n \rvert ((s, t]) r^n.
\]
The last expression goes to zero as $\lvert t - s \rvert \to 0$ by the dominated convergence theorem.
Thus, $\xi(\mathbf{x}, z)_t$ is continuous in the topology of locally uniform convergence on $\mathbb{D}$.

Finally, we observe that $\xi(\mathbf{x}, z)_t$ is of bounded variation and has the expression~\eqref{eq:d_xi} from the relation
\[
\xi(\mathbf{x}, z)_{t} - \xi(\mathbf{x}, z)_s = \sum_{n=1}^{\infty} dx_n((s, t]) z^n = \left( \sum_{n=1}^{\infty} dx_n(\mathord{\cdot}) z^n \right) ((s, t]). \qedhere
\]
\end{proof}

\subsection{Controlled Loewner--Kufarev equation}
\label{sec:LKeq}

Let $(f_t)_{0 \leq t \leq T}$ be a family of holomorphic functions on $\mathbb{D}$.
We temporarily assume that
\begin{enumerate}
\item[$\mathrm{(I)}$] $t \mapsto f_t'(z)$ is measurable for each $z \in \mathbb{D}$, and $(f_t')_{0 \leq t \leq T}$ is locally bounded on $\mathbb{D}$.
\end{enumerate}
Under this assumption, let us suppose that the \emph{controlled Loewner--Kufarev equation}
\begin{equation} \label{eq:LKeq}
f_t(z) - z = \int_0^t z f_s'(z) \, \{ dx_0(s) + d\xi(\mathbf{x}, z)_s \}, \quad t \in [0, T], \ z \in \mathbb{D},
\end{equation}
holds for the driving path $(x_0, \mathbf{x})$ in Section~\ref{sec:driver}.
The assumption~$\mathrm{(I)}$ ensures that the last integral is well-defined.
Moreover, by arguing in the same way as in \cite[Remark~2.2~(b)]{AF18}, we see that
\begin{enumerate}
\item[$\mathrm{(II)}$] $(f_t)_{0 \leq t \leq T}$ is continuous with respect to the locally uniform convergence on $\mathbb{D}$.
\end{enumerate}
From Cauchy's integral formula, this property implies that
\begin{enumerate}
\item[$\mathrm{(III)}$] $(f_t')_{0 \leq t \leq T}$ is continuous with respect to the locally uniform convergence on $\mathbb{D}$,
\end{enumerate}
which is obviously a stronger property than $\mathrm{(I)}$.
Thus, as far as the solutions to \eqref{eq:LKeq} are concerned, the conditions $\mathrm{(I)}$--$\mathrm{(III)}$ are mutually equivalent.

Taking the discussion above into account, we say that a family $(f_t)_{0 \leq t \leq T}$ of holomorphic functions on $\mathbb{D}$ is a \emph{(holomorphic) solution} to the controlled Loewner--Kufarev equation driven by $(x_0, \mathbf{x})$ if it is continuous in the topology of locally uniform convergence and satisfies \eqref{eq:LKeq}.

In Section~\ref{sec:main}, except in Theorem~\ref{thm:uniqueness}, we consider a holomorphic solution to the $\omega$-controlled Loewner--Kufarev equation.
The \emph{control function} $\omega \colon \{\, (s,t) \mid 0 \leq s \leq t \leq T \,\} \to \mathbb{R}_{+}$ is a continuous function with super-additivity
\[
0 \leq s \leq t \leq u \leq T \Rightarrow \omega(s, t) + \omega(t, u) \leq \omega(s, u)
\]
and vanishes on the diagonal, i.e., $\omega(t, t) = 0$ (e.g.\ \cite[Section~2.2]{LQ02}).
The driving functions $x_0$ and $\mathbf{x} = (x_n)_{n \geq 1}$ are assumed to be \emph{controlled by $\omega$}
\cite[Definition~3.2]{AF19}:
for any $n \geq 1$, $1 \leq p \leq n$ and integers $i_1, \ldots, i_p \geq 1$
with $i_1 + \cdots + i_p = n$, the inequalities
\[
\left\lvert e^{n x_0(t)} \int_{0 \leq u_1 < \cdots < u_p \leq t}
e^{-i_1 x_0(u_1)} \, dx_{i_1}(u_1) \cdots e^{-i_p x_0(u_p)} \, dx_{i_p}(u_p) \right\rvert
\leq \frac{\omega(0, t)^n}{n!}
\]
and
\begin{align*}
&\left\lvert e^{n x_0(t)} \int_{0 \leq u_1 < \cdots < u_p \leq t}
e^{-i_1 x_0(u_1)} \, dx_{i_1}(u_1) \cdots e^{-i_p x_0(u_p)} \, dx_{i_p}(u_p) \right. \\
&\phantom{=} \left. - e^{n x_0(s)} \int_{0 \leq u_1 < \cdots < u_p \leq s}
e^{-i_1 x_0(u_1)} \, dx_{i_1}(u_1) \cdots e^{-i_p x_0(u_p)} \, dx_{i_p}(u_p) \right\rvert \\
&\leq \omega(s, t) \frac{\omega(0, T)^{n-1}}{(n-1)!}
\end{align*}
hold for any $0 \leq s \leq t \leq T$.

\section{Main results}
\label{sec:main}

\subsection{Existence, uniqueness and holomorphic extension across the unit circle}
\label{sec:hol_ext}

In this subsection, we prove the existence and uniqueness of a solution to the $\omega$-controlled Loewner--Kufarev equation and, as a byproduct, the fact that the solution can be extended holomorphically across $\partial \mathbb{D}$.
We use the setting in Section~\ref{sec:setting}.

Only in the next theorem, the equation is not assumed to be $\omega$-controlled:
\begin{theorem}
\label{thm:uniqueness}
A solution to the controlled Loewner--Kufarev equation driven by $(x_0, \mathbf{x})$ is unique (if it exists).
\end{theorem}

\begin{proof}
Let $(f_t)_{0 \leq t \leq T}$ be a solution.
As it is continuous and $f_0'(z) = 1$, the quantity $C(t) := f_t'(0)$ is non-zero up to a certain time $\tilde{T} \in (0, T]$.
We can write the Taylor expansion of $f_t$, $t \in [0, \tilde{T})$, around the origin as
\begin{equation} \label{eq:Taylor}
f_t(z) = C(t) (z + c_1(t) z^2 + c_2(t) z^3 + \cdots).
\end{equation}
The coefficients $C(t)$ and $c_n(t)$, $n \geq 1$, are all continuous functions owing to Cauchy's integral formula.
By substituting this expression into \eqref{eq:LKeq}, we have the recurrence relations
\begin{equation} \label{eq:C}
C(t) - 1 = \int_0^t C(s) \, dx_0(s)
\end{equation}
and
\begin{equation} \label{eq:c_n}
c_n(t) = \int_0^t \left\{ \sum_{k=0}^{n-1} (k + 1) c_k(s) \, dx_{n-k}(s) + n c_n(s) \, dx_0(s) \right\}
\end{equation}
for $t \in [0, \tilde{T})$.
Here, we put $c_0(t) := 1$.
The equations~\eqref{eq:C} and \eqref{eq:c_n} are exactly those in \cite[Proposition~2.6]{AF18}.
It follows from the usual iteration method (see e.g.\ \cite[Proposition~0.4.7]{RY99}) that a continuous function $C(t)$ that satisfies \eqref{eq:C} exists uniquely and is given by $C(t) = e^{x_0(t) - x_0(0)}$.
Hence $C(t)$ is non-zero for all $t \in [0, T]$, and \eqref{eq:Taylor}--\eqref{eq:c_n} are valid for all $t$.

Now, we prove the uniqueness of the continuous functions $c_n(t)$, $n \geq 1$, which satisfy \eqref{eq:c_n} by induction.
Let $n \geq 2$ and assume that $c_k(t)$, $1 \leq k \leq n-1$, are unique.
Suppose that there are two continuous functions $c_{1, n}(t)$ and $c_{2, n}(t)$ both satisfying \eqref{eq:c_n}.
Then by taking their difference, we obtain the equation
\[
c_{1, n}(t) - c_{2, n}(t) = n \int_0^t (c_{1, n}(s) - c_{2, n}(s)) \, dx_0(s),
\]
which has a unique solution $c_{1, n}(t) - c_{2, n}(t) \equiv 0$ by \cite[Proposition~0.4.7]{RY99}.
Hence $c_n(t)$ is unique.
The initial case $n=1$ is proved in the same way.

In this way, we have proved the uniqueness of all coefficients $C(t)$ and $c_n(t)$, $n \geq 1$, which implies that of $(f_t)_{0 \leq t \leq T}$.
\end{proof}

The following theorem and lemma are essentially
(but implicitly)
established in
\cite[Corollary~4.4 and Appendix~A.1]{AF19}.
However, for the sake of completeness and the readers' convenience,
we present them with detailed proofs.

\begin{theorem} \label{thm:existence}
Let $\omega$ be a control function with
$
\omega(0, T) <
1/2
$.
Then there exists a solution to the $\omega$-controlled Loewner--Kufarev equation.
\end{theorem}

\begin{proof}
We put $C(t) = e^{x_0(t) - x_0(0)}$ and define $c_n(t)$, $n \geq 1$, by the relation given in \cite[Theorem~2.8]{AF18}.
They are a (unique) solution to the system of equations~\eqref{eq:C} and \eqref{eq:c_n}.
If the series $\sum_{n=1}^{\infty} c_{n-1}(t) z^n$ has convergence radius not less than one, then by reversing the direction of the proof of Theorem~\ref{thm:uniqueness}, the family $(f_t)_t$ defined by \eqref{eq:Taylor} is shown to be a solution to the controlled Loewner--Kufarev equation.

Now by Lemma~\ref{eq:estimate} below,
we
easily see from \eqref{eq:estimate} that $\sum_{n=1}^{\infty} c_{n-1}(t) z^n$ has convergence radius greater than one if
$
\omega(0, T) <
1/2
$.
\end{proof}

Keeping the present notation and the assumption
for $( f_{t} )_{t}$ being a solution to the
$\omega$-controlled Loewner-Kufarev equation,
we have

\begin{lemma} 
\label{eq:estimate}
$
\lvert c_n(t) \rvert
\leq
4^{-1}
n (n+1)
( 2 \omega (0,T) )^{n}
$.
\end{lemma} 
\begin{proof} 

By \cite[Theorem~2.8]{AF18},
\begin{equation*}
\begin{split}
\vert c_{n}(t) \vert
\leq
\sum_{p=1}^{n}
\sum_{\substack{
	i_{1}, \cdots , i_{p} \in \mathbb{N}: \\
	i_{1} + \cdots + i_{p} = n
}}
\widetilde{w} (n)_{ i_{1}, \cdots , i_{p} }
\frac{
	\omega (0,t)^{n}
}{
	n!
},
\end{split}
\end{equation*}
where
$
\widetilde{w} (n)_{ i_{1}, \cdots , i_{p} }
:=
\{ ( n - i_{1} ) + 1 \}
\{ ( n - ( i_{1} + i_{2} ) ) + 1 \}
\cdots
\{ ( n - ( i_{1} + \cdots + i_{p-1} ) ) + 1 \}
\leq
n
(n-1)
\cdots
(n-p)
=
n \binom{n-1}{p} p!
\leq
n 2^{n-1} p!
$.
Therefore we have
\begin{equation*}
\begin{split}
\vert c_{n}(t) \vert
& 
\leqslant
n 2^{n-1}
\frac{
	\omega (0,t)^{n}
}{
	n!
}
\sum_{p=1}^{n}
p!
\sum_{\substack{
	i_{1}, \cdots , i_{p} \in \mathbb{N}: \\
	i_{1} + \cdots + i_{p} = n
}}
1
= 
2^{n-1}
\omega (0,t)^{n}
\left(
	n
	\frac{ 1 }{ n! }
	\sum_{p=1}^{n}
	p!
	\binom{n-1}{p-1}
\right) .
\end{split}
\end{equation*}
The last factor on the right-hand side
can be simply estimated as
\begin{equation*}
n
\frac{ 1 }{ n! }
\sum_{p=1}^{n}
p!
\binom{n-1}{p-1}
=
\sum_{p=1}^{n} \frac{ p }{ (n-p)! }
\leq
\sum_{p=1}^{n} p
=
\frac{ n (n+1) }{ 2 },
\end{equation*}
so that we have
$
\vert c_{n} (t) \vert
\leq
n (n+1) 2^{n-2} \omega (0,t)^{n}
=
4^{-1}
n (n+1)
( 2\omega (0,T) )^{n}
$.

\end{proof} 

In the last line of the proof of Theorem~\ref{thm:existence}, the Taylor series~\eqref{eq:Taylor} of $f_t$ has convergence radius strictly greater than one.
Thus, we obtain the following corollary:

\begin{corollary} \label{cor:hol_ext}
Let $\omega$ be a control function with
$
\omega(0, T) <
1/2
$
and $(f_t)_{0 \leq t \leq T}$ a unique solution to the $\omega$-controlled Loewner--Kufarev equation.
Then for each $t \in [0, T]$, the function $f_t$ can be extended holomorphically to an open neighbourhood of $\overline{\mathbb{D}}$.
\end{corollary}

\subsection{Univalence and starlikeness on the unit disk}
\label{sec:univalence}

In this subsection, we prove that the solution to the $\omega$-controlled Loewner--Kufarev equation is univalent and starlike.
Let $\alpha$ be the smallest real solution to the quartic equation
$
2x^{4} - 8x^{3} + 11 x^{2} - 10 x + 2 = 0
$.
Note that
$
\alpha/2 \approx 0.13105 \cdots \in ( \frac{1}{8},\frac{1}{7} )
$.

\begin{theorem}
\label{thm:univalence}
Let $\omega$ be a control function with
$
\omega(0, T) \leq
\alpha/2
$
and $(f_t)_{0 \leq t \leq T}$ a unique solution to the $\omega$-controlled Loewner--Kufarev equation.
Then the function $f_t$ is univalent
and starlike
on $\mathbb{D}$ for each $t \in [0, T]$.
\end{theorem}

\begin{proof}
From Section~9 of Alexander~\cite{Al18}, a normalized holomorphic function $f(z) = z + a_2 z^2 + a_3 z^3 + \cdots$ is univalent
and starlike
on $\mathbb{D}$ if $\sum_{n=2}^{\infty} n \lvert a_n \rvert \leq 1$.
(Readers can also find this fact in
Example~2.2 in Chapter~2 of Pommerenke~\cite{Po75} or in Exercise~24 in Chapter~2 of Duren~\cite{Du83} for instance.)
We apply this sufficient condition to the series $z + \sum_{n=2}^{\infty} c_{n-1}(t) z^n$. Here, the $c_n(t)$'s are defined as in the proof of Theorem~\ref{thm:existence}.
In this case, the sufficient condition for the univalence is
\begin{equation} \label{eq:univalence}
\sum_{n=2}^{\infty} n \lvert c_{n-1}(t) \rvert \leq 1.
\end{equation}
Noticing that
$2 \omega (0,T) \leq \alpha < 1$
and by using Lemma~\ref{eq:estimate}, the
left-hand side of this inequality is estimated as follows:
\begin{equation*}
\sum_{n=2}^{\infty} n \lvert c_{n-1}(t) \rvert
\leq
\frac{1}{4} \sum_{n=2}^{\infty} n^{2} (n-1) (2 \omega(0, T))^{n-1}
=
\frac{1}{4}
\cdot
\frac{
	2 ( 2 \omega(0, T) ) ( 2 \omega(0, T) + 2 )
}{
	( 2 \omega(0, T) - 1 )^{4}
}.
\end{equation*}
We can easily check that the last fraction is not greater than one
if and only if
$
\omega(0, T) \leq
\alpha/2
$.
Hence \eqref{eq:univalence} holds under the present assumption.
\end{proof}

Combining Theorems~\ref{thm:uniqueness}, \ref{thm:existence}, \ref{thm:univalence} and Corollary~\ref{cor:hol_ext} yields our goal Theorem~\ref{thm:main}.

\section*{Acknowledgements}
We thank for the constructive comments we received and which helped us to improve and strengthen our presentation.
The third author was supported by JSPS KAKENHI Grant Number JP19J13031.

\end{document}